\documentclass[11pt, letterpaper]{article} % Sets article for letter paper with 11pt font
\usepackage[letterpaper,margin=0.75in]{geometry}
\usepackage[utf8]{inputenc} % For modern character encoding
\usepackage{graphicx}
\usepackage{amsmath}
\usepackage{amsthm}
\usepackage{times}
\usepackage{helvet} \usepackage{courier}  
\newtheorem{theorem}{Theorem}[section]
\newtheorem{remark}{Remark}

\newtheorem{example}{Example}[section]

\title{Koenigs functions in the subcritical and critical Markov branching processes with Poisson probability reproduction of particles }

\begin{document}

\author{
	Penka Mayster$^{1}$ \\
	Assen Tchorbadjieff$^{1,2}$\thanks{Corresponding author. ORCID: 0000-0001-9322-262X}
}

\date{}

\maketitle

\begin{center}
	$^{1}$Institute of Mathematics and Informatics, Bulgarian Academy of Sciences, \\
	Acad. G. Bonchev Str., Bloc 8, 1113 Sofia, Bulgaria \\
	\vspace{0.2cm}
	$^{2}$Centre of Excellence in Informatics and Information and Communication Technologies, \\
	1113 Sofia, Bulgaria \\
	\vspace{0.2cm}
	Email: penka.mayster@math.bas.bg, atchorbadjieff@math.bas.bg
\end{center}

\maketitle

\begin{abstract}
Special functions have always played a central role in physics and in mathematics, arising
as solutions of nonlinear differential equations, as well as in the theory of branching processes, which extensively uses probability generating functions. The theory of iteration of real functions leads to limit theorems for the discrete-time and real-time Markov branching processes.
The Poisson reproduction of particles in real time is analysed through the integration of the Kolmogorov equation. These results are further extended by employing graphical representations of Koenigs functions under subcritical and critical branching mechanisms. The limit conditional law in the subcritical case and the invariant measure for the critical case are discussed, as well. The obtained explicit solutions contain the exponential Bell polynomials and the modified exponential-integral function $\rm{Ein} (z)$.
\end{abstract}

\section{Introduction}

The theory of branching processes describes the stochastic evolution of particle systems by the phenomena of reproduction, multiplicity, and extinction. Key applications in physics, particularly in high-energy particle interactions and neutron chain-reactions, are developed in the works \cite{H}, \cite{P}, \cite{DCh}. There are many applications in biology, as well,  \cite{MC},  \cite{JP},  \cite{KM}, \cite{HJV}. 

The Galton-Watson (GW) branching is a discrete-time process, which models the number of particles in each generation. This stochastic process, denoted $Z_n$ for $n=0,1,2,...$ begins with $Z_0=1$ and evolves according to the recurrence relation:
$$Z_{n+1}=\sum^{Z_n}_{j=1}\xi_j,\quad Z_0=1,\quad Z_1=\xi, $$
where $ \xi _{j}$ is a set of independent and identically-distributed (iid) random variables (rv), following the same probability distribution as the random variable $\xi:=Z_1 $. Each $\xi$ represents the offspring number of a single particle, and reproduction occurs independently among particles. In particular, the discrete-time model does not consider particle lifetimes. The probability generating functions (pgf) of the sequence of random variables  $ (Z_1, Z_2,..., Z_n,...)$ are related by the iterative equations:
$$H_1(s)=H(s),\quad H_{2}(s)=H(H(s)),\quad H_{n+1}(s)=H(H_{n})(s),\quad |s|\leq 1. $$
The theory of such iterative functional equations is studied in the book \cite{Ku}.
The iteration of the exponential function - specifically the pgf of the Poisson distribution - has attracted the attention of  Paul L\'{e}vy  (1927) and E.M. Wright (1947), see \cite{W} and references therein.
The problem of the iteration of an analytic function is closely related to that of finding an analytic solution of the equations due to Abel and Schr\"{o}der.
The imbeddability of discrete time composition semigroup, $(H_{n}(s), n=1,2,...)$ to the continuous time semigroup
$$F(t+\tau,s)=F(t,F(\tau,s)),\quad t>0,\quad \tau>0,\quad F(0,s)=s,$$
is considered in \cite{KMc} and developed in \cite{El} and \cite{VG}.
The Schr\"{o}der's and Abel's, \cite{Sch}, equations appear in a natural way
in the theory of branching processes, \cite{H}, where the solution of the Kolmogorov equation is expressed by the compositional inverse of the Koenigs function, \cite{Ko}.
The Schr\"{o}der's equation, is an eigenvalue equation for the composition-function operator.
Thus,  the eigenvalue is given by $ M(t)=F'_s(t,1)$. The eigenvector is known as Koenigs function.

The Markov branching process (MBP) $X(t), t \geq 0, X(0)=1,$ is a continuous-time model studying the evolution of a population in real time $t>0 $ under the assumption that the lifetime of particles is exponentially distributed with parameter $K>0$, see \cite{AN}, \cite{S}. This way, in any positive time $t>0$ there is a family of particles of all generations with positive probability.
The offspring number is denoted by the integer-valued rv $\eta$ with the following pgf, $h(s)=E[s^\eta]$.
The first derivative of the reproduction pgf $ h'(1)=m$ is the basic parameter in the theory of branching processes.
The constant $m:=E[\eta]$ defines the offsprings mean. The reproduction is classified as subcritical if $0< m<1$, respectively, critical and supercritical, if $m=1$ or $m>1$.
The ultimate extinction probability
\begin{equation}
	q:=\lim_{t\rightarrow \infty} \textbf{P}(X(t)=0)\label{qq}
\end{equation}
is defined by the minimal solution of the equation
\begin{equation}
	h(s)=s,\quad s\geq 0, \quad h(1)=1.\label{q}
\end{equation}
Remark, that for the critical and subcritical cases the value $s=1$ is the minimal solutions of $h(s)=s$ and $F(t,s)=s$. For the supercritical case, the ultimate extinction probability (\ref{qq}), and (\ref{q}) and the value $s=1$, are the fixed points, (known as Denjoy-Wolfs points), for
$$F(t,1)=1,\quad F(t,q)=q,\quad t>0,\quad 0\leq q \leq1. $$
The pgf of the number of particles alive at the time $t>0$ defined as
\begin{equation}
	F(t,s)=\sum^\infty_{k=0}s^k P(X(t)=k|X(0)=1)\label{F}
\end{equation}
yields the backward Kolmogorov equation with initial condition
\begin{equation}
	\frac{\partial}{\partial t}\left(F(t,s)\right)=f\left(F(t,s)\right),\quad f(s)=K(h(s)-s),\quad F(0,s)=s. \label{CKs}
\end{equation}
The infinitesimal generating function of the MBP defined by (\ref{F}) is given by the function $ f(s)$.
In particular, for the linear birth-death process, geometric and Poisson reproduction,
$$ h(s)=1-p+ps^2,\quad h(s)=\frac{1}{1+m-ms},\quad h(s)=e^{-\lambda(1-s)} .$$
The equation (\ref{CKs}) is nonlinear, of the type separate differentials, due to the time-homogeneous property (Fig. \ref{fig:hs}).
\begin{figure}[t]
	\includegraphics[scale=.5]{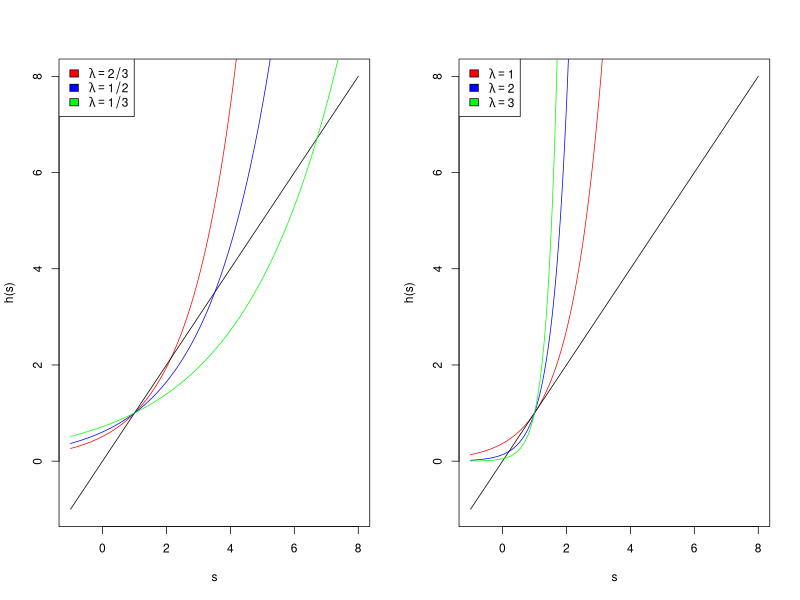} %Bern(.eps,.pdf)
	\caption{%
		Graphics of h(s) for sub-, critical and supper- critical branching processes. The subcritical BP of $\lambda = 1/3, 1/2, 2/3$ is plotted on the left side. The critical ($\lambda=1$) and supercritical ($\lambda=2, 3$) are on the right side.
	}\label{fig:hs} %% no full stop at the end
\end{figure}
It is equivalent to the following ordinary differential equation
\begin{equation}
	\frac{f'(1) dx}{f(x)}=f'(1)dt,\quad x=F(t,s),\quad F(0,s)=s. \label{Ks}
\end{equation}
The implicit solution for the pgf $F(t,s) $ to the Kolmogorov equation (\ref{Ks}, \ref{CKs})  is defined by integral
$$\int \frac{ dx}{h(x)-x}. $$
The bounds  of integration must be defined in consideration of the real roots of the equation $h(s)= s$,  and the domain of definition for the pgf  $h(s)$
following the main characteristics of reproduction, such as the offspring's mean $ m=h'(1)$ and ultimate extinction probability $0<q<1$.
The mathematical expectation of the number of particles alive at the positive time $ t>0 $ has the following exponential decreasing, increasing, or constant behaviour,
$$ E[X(t)]:=M(t)=\exp\{f'(1)t\}, \quad t>0 , \quad f'(1)=K(m-1)<0,=0,>0 .$$
The discrete-time GW branching process generated by a geometric distribution is imbeddable in a continuous-time MBP with a quadratic infinitesimal generating function, i.e., the so-called \textbf{birth}-death process and Yule-Furry process, see \cite{KMc}. The continuous-time MBP with geometric branching is studied in \cite{TMT}, \cite{TM}, \cite{TMM}. The inverse of the Koenigs function in non-critical cases are expressed by the Fox-Wright function $_1\Psi_1(\alpha,a;\beta,b;s) $,  and by Lambert-W function in the critical case. The series expansion of the solution $F(t,s)$ in the neighbourhood of the point $s=1$ is based on the hyperharmonic  numbers.

The discrete-time GW branching process generated by the Poisson distribution can not be imbedded in a continuous-time MBP, see \cite{KMc}. This is our motivation to study the Poisson branching reproduction in real time.

\section{Poisson probability reproduction in real time}
The pgf of the Poisson law and the mathematical expectation are defined as
$$ h(s)=e^{-\lambda(1-s)}=e^{\lambda(s-1)},\quad h'(s)=\lambda h(s),\quad  |s|\leq 1 ,\quad m=h'(1)=\lambda.$$
The  Denjoy-Wolfs points solving the equation $ s=e^{-\lambda(1-s)}$ are defined by the Lambert-W function as follows,
$$ -\lambda s e^{-\lambda s}=(-\lambda) e^{-\lambda}, \quad (-\lambda) q=W(-\lambda e^{\lambda}),\quad -1=W(-e^{-1}). $$
The infinitesimal generating function $f(s)$ and its reciprocal are writen as
$$f(s)=K(e^{-\lambda (1-s)}-s),\quad \frac{K}{K(e^{-\lambda}e^{\lambda s}-s)}= \frac{e^\lambda}{1+(e^{\lambda s} -1-se^{\lambda})}. $$
The implicit solution for the pgf $F(t,s) $ to the Kolmogorov equation is written in the form of the Abel's and Schr\"{o}der's equations. We introduce the main notations for the Koenigs functions  respectively following the critical parameter $\lambda<1, \lambda=1,  \lambda>1$.

\textit{Subcritical case} Let $0<\lambda<1$ and $ 0<s<1$.
The Abel's equation is written as follows,
$$ A(F(t,s))=f'(1)t +A(s),\quad f'(1)=K(\lambda-1)<0,\quad A(s)=\int^s_{x=0} \frac{ (\lambda -1)dx}{e^{-\lambda}e^{\lambda x}-x},$$
and in the neighborhood  of the point $s=1$ we consider
$$\frac{ f'(1)}{f(x)}=\frac{1}{(x-1)(1-g(x))},\quad g(x)=\frac{\lambda}{1-\lambda}\sum^\infty_{n=1}\frac{(\lambda)^n}{n+1} \frac{(x-1)^n}{n!}.$$
The Schr\"{o}der's equation is written as follows,
$$B(F(t,s))=e^{K(\lambda-1)}B(s), \quad B(s)=e^{A(s)}. $$
The Koenigs function $B(s)$
yields the following relation on their logarithmic derivative
\begin{equation}
	\frac{B'(s)}{B(s)}=\frac{\lambda-1}{h(s)-s}<0,\quad B(0)=1,\quad B(1)=0. \label{BB}
\end{equation}
The function $F_\ast(s)=1-B(s)$ is a pgf of the limit conditional law (LCL). The probability mass function of the LCL is defined by
derivatices $F^{(n)}_\ast(0)=-B^{(n)}(0)$.

\textit{Critical case} Let $\lambda=1, h(s)=e^{s-1}$ and $ 0<s<1$.
The Abel's equation is written by the function $ U(s)$,
$$ U(F(t,s))=Kt +U(s),\quad U(s)=\int^s_{x=0} \frac{K dx}{f(x)}=\int^s_{x=0} \frac{dx}{e^{x-1}-x},\quad f'(1)=0,$$
and in the neighborhood  of the point $s=1$ we consider
\begin{equation}
	\frac{ K}{f(x)}=\frac{2}{(x-1)^2(1-2\varphi(x))},\quad \varphi(x)=\sum^\infty_{n=1}\frac{1}{(n+1)(n+2)} \frac{(x-1)^n}{n!}. \label{C}
\end{equation}
The Schr\"{o}der's equation is written as follows,
$$C(F(t,s))=e^{Kt}C(s), \quad C(s)=e^{U(s)}. $$
The function $U(s)$ is a power series of nonnegative coefficients. The sequence $$\frac{U^{(n)}(0)}{n!},\quad n=0,1,2,..., $$
represents the stationary measure for $X(t), t>0$, see Harris, p.102-112, \cite{H} and Pakes \cite{AP}.

\textit{Supercritical case} Let $\lambda>1$ with $$ q<s<1,\quad q=e^{\lambda(q-1)}, \quad f'(q)= K(q \lambda-1)<0,$$
and
$$\frac{ f'(q)}{f(x)}=\frac{1-q\lambda}{(1-\lambda)(x-1)(1-g(x))},\quad g(x)=\frac{\lambda}{1-\lambda}\sum^\infty_{n=1}\frac{\lambda^n}{n+1} \frac{(x-1)^n}{n!}.$$
where $$\frac{1-q\lambda}{(1-\lambda)}<0,\quad \lambda>1,\quad q \lambda<1,\quad 0<q<\frac{1}{\lambda}<1.$$
and the Koenigs function $Q(s)$ when  $0<q<1$ is
$$\lambda>1, \quad \log Q(s)=\int^s_{x=q} \frac{ f'(q)dx}{f(x)},\quad f'(q)=K(q\lambda-1)<0,\quad 0<q=e^{-\lambda(1-q)}<1. $$
We study in parallel, and in comparison, the subcritical and critical MBP. The supercritical case merits special consideration for asymptotic,
as for the geometric reproduction, \cite{TMM}.

The explicit solutions for the MBP with Poisson reproduction contain the exponential Bell polynomials, \cite{Fe}, \cite{WW}, and modified exponential-integral function $\rm{Ein} (z)$. The function
$\rm{Ein}(x)$ is an entire function, related to the exponential generating function of the harmonic numbers $H_n$:
$$ \rm{Ein}(x)=e^{-x}\sum^\infty_{n=1}H_n \frac{x^n}{n!}, \quad H_n=1+\frac{1}{2}+\frac{1}{3}+....\frac{1}{n} .$$
The special functions exponential-integral $ \rm{Ei}(x), x \neq 0,$ and modified exponential integral $ \rm{Ein}(x)$ are defined as
$$\rm{Ei}(x)=\int^x_{-\infty}\frac{e^y dy}{y}, \quad \rm{Ein}(x)=\int^x_0\frac{(1-e^{-y})dy}{y}=\sum^\infty_{k=1}\frac{(-1)^{k+1} x^k}{k.k!},$$
$$ \rm{Ei} (x)=\gamma+\log (|x|) - \rm{Ein} (-x).$$
We remark the following relation with the famous Euler-Mascheroni constant, given by the value of Digamma function
$$\gamma=-\psi(1):=-\Gamma'(1) \approx 0.5772156649015328606065. . ..$$
The number $e$, that is the base of the natural logarithm and exponential function, alternatively, $e$ can be called Napier's constant after John Napier,
$e=2.7182818284... $,  (see \cite{JC}).

The following sub-section $2.1.$ gives the numerical evaluation and graphics of the Koenigs functions $ B(s)$ and $C(s) $ based on the recurrence relations to their derivatives at $s=0$ and $s=1$.
In Section $3$, we study the solutions of the Abels equation $A(s) $  presented by integral of the reciprocal power series in the neighbourhood of the point $s=0$.
In the sub-section $3.2.$, the solutions of the Koenigs function $B(s)$  in the neighborhood of the point $s=1$ is presented in the form
$$B(s)=(1-s)e^{G(s)},\quad G(s)=\sum^\infty_{n=1}\frac{(\lambda(s-1))^n-(-\lambda)^n}{n}\sum^n_{k=1}\left(\frac{\lambda}{1-\lambda}\right)^k \frac{k!}{n!}B_{n,k}(g_\bullet) $$
where the sequence
$$g_\bullet=\left(g_1=\frac{1}{2}, g_2=\frac{1}{3},..., g_n=\frac{1}{n+1},...\right). $$
The Section $4$ is devoted  to the critical case.  In the neighborhood of the points $ s=0$, the solution is represented as $ C(s)=e^{U(s)}$, and in the neighborhood of the point $ s=1$, as integral of  (\ref{C} ).

\subsection{Recurrence relations for the Koenigs functions derivatives when $0<\lambda<1 $ and $\lambda=1,$ }
The logarithmic derivative (\ref{BB})  contains the main information for the behaviour of the MBP. The consecutive derivatives of them, bring us the recurrence relation, as follows.

$\bullet $	Let $ 0<\lambda<1$.
For the Poisson reproduction law, in the neighborhood of the point $s=0$, the pgf of branching is
$$h(s)=e^{-\lambda}e^{\lambda s}=e^{-\lambda}\sum^\infty_{k=0}\frac{\lambda^k}{k!}s^k, \quad h(0)=e^{-\lambda}, \quad h'(0)=e^{-\lambda}\lambda, $$
 $$ h^{(k)}(0)=e^{-\lambda}\lambda^k,\quad k=2,3,...,$$
and the recurrence relation is specified as
\begin{equation}\label{eq:BS}
	B^{(n+1)}(0)e^{-\lambda}=B^{(n)}(0)e^{-\lambda}\{(\lambda+n-1)e^\lambda-n\lambda\}-\sum^n_{k=2}\binom{n}{k}B^{(n+1-k)}(0)e^{-\lambda}\lambda^k.
\end{equation}
We calculate the Taylor series expansion in the neighborhood of $s=0$ with $ B(0)=1$ as
$$ B(s)=1+\sum^\infty_ {n=1}B^{(n)}(0)\frac{s^n}{n!},\quad B'(0)=(\lambda-1)e^\lambda<0, \quad B''(0)= B'(0)\lambda\{e^\lambda -1\},$$
$$B^{(3)}(0)=B'(0)\lambda \{(1+\lambda)e^{2\lambda}-(1+3\lambda)e^\lambda +\lambda \}, ....$$
The function $B(s)$ is decreasing, concave and all its derivatives $B^{(k)}(s)$ are negative in the interval, $ 0<s<1$ (see fig. \ref{fig:Bs}). It is worth noting that for $ 0<\lambda <1/6$, the graphic of the function  $B(s)$ is  almost equal to $1-s$. 
\begin{figure}[t]
	\includegraphics[scale=.5]{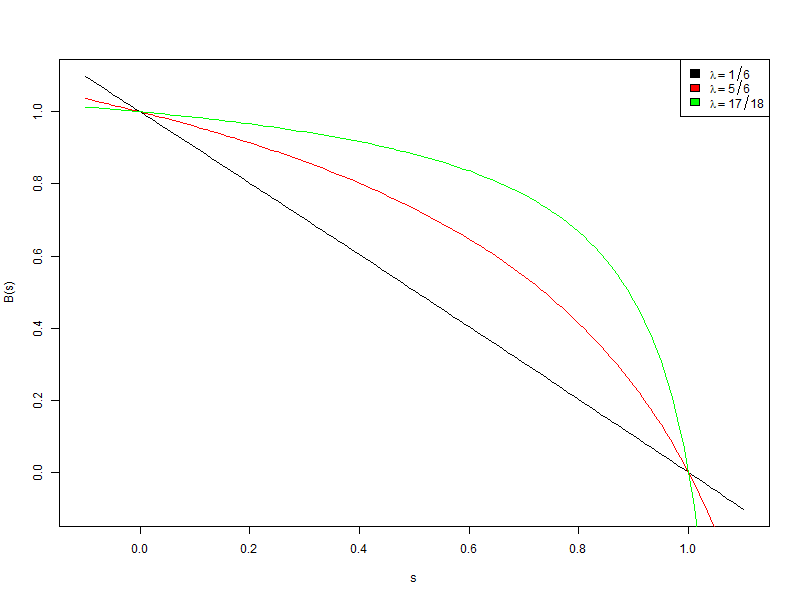} %Bern(.eps,.pdf)
	\caption{%
		Graphics of the function $B(s)$ using eq. \ref{eq:BS} for $\lambda =1/6$ (black), $\lambda =5/6$ (red) and $\lambda =17/18$ (green).  
	}\label{fig:Bs} %% no full stop at the end
\end{figure}

The pgf of the LCL, $F_\ast (s) = 1-B(s)$, is increasing, convex, and defines the probability mass function  $(f_1, f_2,...,) $ as follows,
$$F_\ast (0)=0,\quad F_\ast(1)=1,\quad F_\ast(s)=\sum_{n=1}^{\infty}f_n s^n,\quad  f_n=\frac{-B^{(n)}(0)}{n!},$$
$$ f_1= (1-\lambda)e^\lambda,\quad f_2= \lambda (e^\lambda-1)\frac{f_1}{2},....$$
(see Figure \ref{fig:Pstar})

\begin{figure}[t]
	\includegraphics[scale=.5]{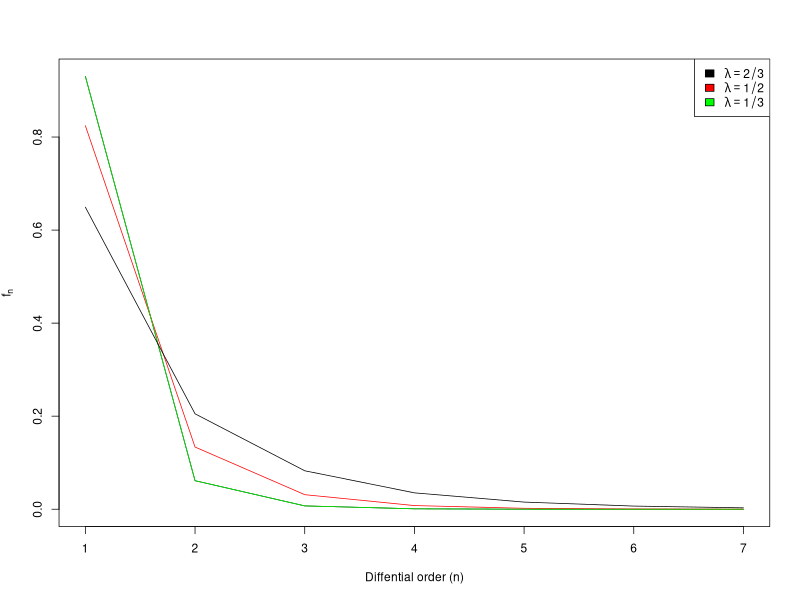} %Bern(.eps,.pdf)
	\caption{%
		Histogram of $ f_n=\frac{-B^{(n)}(0)}{n!}$ for different $\lambda$.
	}\label{fig:Pstar} %% no full stop at the end
\end{figure}

For the subcritical Poisson reproduction law, $0<\lambda <1$, the series expansion in the neighbourhood the point $s=1$
$$h(s)=e^{\lambda (s-1)}=\sum^\infty_{k=0}\frac{\lambda^k}{k!}(s-1)^k, \quad h^{(k)}(1)=\lambda^k. $$
We start with $B(1)=0$. Then, using the representation introduced in the following section
$$ B(s)=(1-s)e^{G(s)},\quad B'(s)=-e^{G(s)}+(1-s)G'(s),  $$
we define
$$B'(1)=-e^{G(1)}<0,\quad G(1)=\sum^\infty _{n=1}\frac{(-1)^{n-1}\lambda^n}{n}\sum^n_{k=1}\left(\frac{\lambda}{1-\lambda}\right)^k \frac{k! }{n!}B_{n,k}(g_\bullet)>0.$$
And for $ n=2,3,...$ we have $h^{(k)}(1)=\lambda^k$ 	and ,
$$B^{(n)}(1)(n-1)(1-\lambda)=\sum^n_{k=2}\binom{n}{k}B^{(n+1-k)}(1)\lambda^k, $$
$$n=1,\quad B'(1)(\lambda-\lambda)=0,\quad B'(1)=-e^{G(1)}<0,\quad B(0)=e^{G(0)},\quad G(0)=0,$$
$$n=2,\quad B''(1)(1-\lambda)=B'(1)\lambda^2,\quad B''(1)=B'(1) \frac{\lambda^2}{1-\lambda}, ...,.$$

$\bullet$ Let $ \lambda=1$.
The Koenigs function $C(s)$ in the critical MBP
yields the following relation on their logarithmic derivative
\begin{equation}
	\frac{C'(s)}{C(s)}=\frac{1}{h(s)-s}>0,\quad \lambda=1,\quad C(0)=1,\quad C(1)=\infty, \label{CC}
\end{equation}
For the critical Poisson reproduction law, $\lambda=1$, in the point $s=0$, the recurrence relation based on (\ref{CC}) is
$$C^{(n+1)}(0)=C^{(n)}(0)\{(n+1)e-n\}-\sum^n_{k=2}\binom{n}{k}C^{(n+1-k)}(0),\quad C'(0)=e, $$	
and in the neighborhood of the point $s=1$ we have some indefinite relation,
$$C^{(n+1)}(1)\{e- e\}=e\left\{C^{(n)}(1)-\sum^n_{k=2}\binom{n}{k}C^{(n+1-k)}(1)\right\}. $$
In the critical case, all derivatives tend to infinity in the neighbourhood of $s=1$, where
$$C(1)=\infty,\quad C'(1)=\infty,\quad C^{(n)}(1)=\infty, \quad n=2, 3,....$$
The function $C(s)$ is increasing and  convex on the interval $0<s<1$.
\section{Subcritical MBP $X(t)$, $0<\lambda<1$, }

\subsection{Integration of the reciprocal power series in the neighbourhood of $s=0$}

Let us denote the function $A(s)$ in the neighbourhood of the point $s=0$, as follows,
$$A(s):=\log B(s),\quad B(s)=\exp(A(s))>0,\quad A(s):=\int^s_{x=0} \frac{( \lambda-1) dx}{h(x)-x},\quad 0\leq s\leq 1. $$
The reciprocal power series is calculated in the following notations,
$$\frac{1}{h(s)-s}=\frac{e^\lambda}{e^{\lambda s}-se^\lambda}=\frac{e^\lambda}{1+p(s)}=e^\lambda\{1+a(s)\}. $$
The functions $p(s)$ and $a(s)$ are defined by their derivatives at the point $s=0$,
$$p(s)=e^{\lambda s}-(1+se^\lambda)=\sum^\infty_{n=1}p_n\frac{s^n}{n!},\quad a(s)= \sum^\infty_{n=1}a_n\frac{s^n}{n!},\quad p(0)=0,\quad a(0)=0.$$
The general Leibniz rule applied on the equality $\{1+p(s)\}\{1+a(s)\}=1 $ gives the following recurrence relation for derivatives at the point
$s=0$ in the following notations, $p_0=1+p(0)=1$ and $a_0=1+a(0)=1$, $n=1, 2, 3,...,$
$$\sum^n_{k=0}\binom{n}{k}p_k a_{n-k}=0,\quad p_0=a_0=1,\quad p_1=\lambda-e^\lambda<0, \quad p_2=\lambda^2,...,\quad p_n=\lambda^n,... .$$
In particular,
$$a_1=-p_1=e^\lambda-\lambda>0,\quad a_2=-p_2 +2(p_1)^2,\quad a_3=-\lambda^3+3!p_1 \lambda^2-3!(p_1)^3. $$
Also, the reciprocal power series is represented by the exponential Bell polynomials over the function
$$p(s)=(\lambda-e^\lambda)s+\lambda^2 \frac{s^2}{2!}+ \lambda^3 \frac{s^3}{3!}+...,$$
as follows,
$$\frac{e^\lambda}{1-(-p(s))}= e^\lambda \left(1+\sum^\infty_{n=1}\sum^n_{k=1}B_{n,k}(p_\bullet)(-1)^k k!\frac{ s^n}{n!}\right).$$
The function $p(s)$ is the exponential generating function of the sequence
$$ p_\bullet=(\lambda-e^\lambda, \lambda^2,\lambda^3,...)=(\lambda,\lambda^2,...)+(-e^\lambda,0,0,...) .$$
The exponential Bell polynomials of the sum of two sequences is defined by
$$B_{n,k}(x+y)=\sum^k_{j=0}\sum^n_{i=0}\binom{n}{i}B_{i,j}(x)B_{n-i,k-j}(y). $$
In particular, when the sequence $ y=(-e^\lambda,0,0,...)$, then
$$B_{n-i,n-j}(y)=(-e^\lambda)^{n-i},\quad n-i=k-j,  $$
otherwise, $ B_{n-i,n-j}(y)=0$, (see \cite{CM}).

Then, we have two equivalent representations,
$$-a_n=\sum^n_{k=1}\binom{n}{k}p_k a_{n-k},\quad a_n=\sum^n_{k=1}B_{n,k}(p_\bullet)(-1)^k k!.$$
\begin{example}\label{Example1} For instance, the table of $ B_{n,k}(p_\bullet), n=5$ for any values $0<\lambda\leq 1$,  is given by
\begin{equation}
\label{eq:lambda-matrix}
\begin{pmatrix}
\lambda - e^{\lambda} & 0 & 0 & 0 & 0 \\
\lambda^{2} & (\lambda - e^{\lambda})^{2} & 0 & 0 & 0 \\
\lambda^{3} & 3(\lambda - e^{\lambda})\lambda^{2} & (\lambda - e^{\lambda})^{3} & 0 & 0 \\
\lambda^{4} & 4(\lambda - e^{\lambda})\lambda^{3} + 3\lambda^{4} 
& 6(1 - e^{\lambda})^{2}\lambda^{2} & (\lambda - e^{\lambda})^{4} & 0 \\
\lambda^{5} & 5(\lambda - e^{\lambda})\lambda^{4} + 10\lambda^{5} 
& 10(1 - e^{\lambda})^{2}\lambda^{3} + 15(1 - e^{\lambda})\lambda^{4} 
& 10(1 - e^{\lambda})^{3}\lambda^{2} & (\lambda - e^{\lambda})^{5}
\end{pmatrix}
\end{equation}
\end{example}

The integration of the series expantion, takes form
$$A(s)=(\lambda-1)e^\lambda\int^s_{x=0}\left\{1+\sum^\infty_{k=1}a_k \frac{x^k}{k!} \right\}dx =(\lambda-1)e^\lambda \left\{s+\sum^\infty_{k=2}a_{k-1} \frac{s^{k} }{k!} \right\}. $$
It means,
$$A(0)=0,\quad A'(0)=(\lambda-1)e^\lambda<0 ,\quad A^{(2)}(0)=(\lambda-1)e^\lambda (e^\lambda-\lambda)<0,..., $$
$$ A^{(n)}(0)= (\lambda-1)e^\lambda  a_{n-1}<0,\quad n=3, 4,....$$
Obviously, the function $A(s)$ is negative, decreasing in the interval $ 0<s<1$ and $\lim_{s\rightarrow 1}A(s)=-\infty $. The graphical representation is available in Figure \ref{fig:AS}.
\begin{figure}[t]
	\includegraphics[scale=.5]{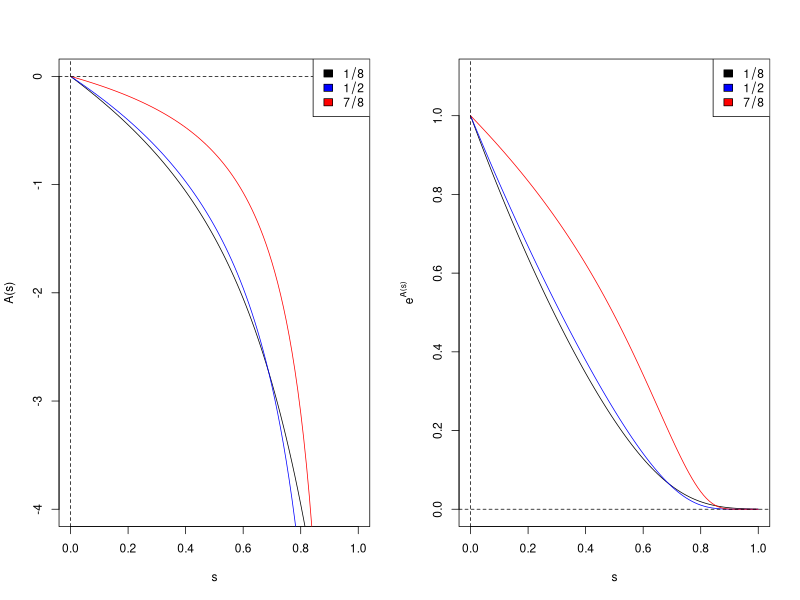} %Bern(.eps,.pdf)
	\caption{%
		Graphics of $A(s)$ (left) and $B(s)=e^{A(s)}$ (right).
	}\label{fig:AS} %% no full stop at the end
\end{figure}

Using the Faa Di Bruno formula, we confirm the derivative of the Koenigs function $B(s)$ from the previous section.
$$B(s)=\exp(A(s)),\quad B^{(n)}(s)=\exp(A(s))\sum^n_{k=0}B_{n,k}(A'(s), A''(s),...). $$
In particular, $B(1)=0$ and $B(0)=1,$
$$ B'(0)=A'(0),\quad B''(0)=(A'(0))^2+A''(0),\quad B^{(3)}(0)=(A'(0))^3+3A'(0)A''(0)+A^{(3)}(0),.... $$
\begin{theorem}
	Let the function $$A(s)=\int^s_{x=0} \frac{ (\lambda -1)dx}{e^{-\lambda}e^{\lambda x}-x}.$$
	The explicit form of the function $A(s)$ for $0<\lambda<1$ is represented by the series expantion
	$$A(s)=(\lambda-1)e^\lambda\sum^\infty_{k=0}(-1)^k\sum^k_{j=0}\binom{k}{j}{e^{\lambda j s}}(-1)^{k-j} \sum^{k-j}_{i=1}\left(\frac{e^{\lambda} }{\lambda j}\right)^i (1+se^\lambda)^{k-j-i}[k-j]_{i\downarrow}$$
	$$-(\lambda-1)e^\lambda\sum^\infty_{k=0}(-1)^k\sum^k_{j=0}\binom{k}{j}(-1)^{k-j} \sum^{k-j}_{i=1}\left(\frac{e^{\lambda} }{\lambda j}\right)^i [k-j]_{i\downarrow}.$$
	The  increasing and decreasing factorials, are defined by
	$$[x]_{k\downarrow}=\frac{\Gamma(x+1)}{\Gamma(x+1-k)},\quad [x]_{k\uparrow}=\frac{\Gamma(x+k)}{\Gamma(x)}. $$
\end{theorem}

\begin{proof}
	In its explicit form, the function
	$$p(x)=e^{\lambda x}-(1+xe^\lambda) ,\quad (p(x))^k=\sum^k_{j=0} e^{\lambda j x} \binom{k}{j}(-1)^{k-j}(1+xe^\lambda)^{k-j}. $$
	Then the integral
	$$\log(B(s))=(\lambda-1)e^\lambda\sum^\infty_{k=0}(-1)^k\sum^k_{j=0}\binom{k}{j}\int^s_{x=0} e^{\lambda j x}(-1)^{k-j} (1+xe^\lambda)^{k-j}dx $$
	The integration by parts starting with
	$$\int^s_{x=0} e^{\lambda j x}(1+xe^\lambda)^{k-j}dx=\int^s_{x=0}(1+xe^\lambda)^{k-j}d \left(\frac{e^{\lambda j x} }{\lambda j}\right)$$
	shows the proofs.
\end{proof}
\subsection{Subcritical MBP in the neighborhood of $s=1$}
In the subcritical case, $0<\lambda<1$ and $ 0<s<1$,
$$\frac{ f'(1)}{f(x)}=\frac{\lambda-1}{e^{\lambda (x-1)}-x}=\frac{1}{(x-1)(1-g(x))},$$
where
$$g(x)=\frac{e^{\lambda(x-1)}-1-\lambda(x-1)}{(1-\lambda)(x-1)}=\frac{\lambda}{1-\lambda}\left(\frac{e^{\lambda( x-1)}-1}{\lambda(x-1)} -1\right)=\frac{\lambda}{1-\lambda}\sum^\infty_{n=1}\frac{(\lambda)^n}{n+1} \frac{(x-1)^n}{n!}, $$
$$g(0)=\frac{1-\lambda -e^{-\lambda}}{1-\lambda}=\frac{1}{1-\lambda}\sum^\infty_{j=2}\frac{(-\lambda)^j}{j!} ,\quad g(1)=0 . $$
The Taylor series expansion in the neighborhood of $s=1$ is denoted by its derivatives as
$$g(x)=\sum^\infty_{n=1}g^{(n)}(1) \frac{(x-1)^n}{n!},\quad  g^{(n)}(1)=\frac{\lambda}{1-\lambda}\frac{\lambda^n}{n+1},\quad g_n=\frac{1}{n+1}. $$
\textbf{The first method} is based on using the combinatorics identity and Bell polynomials for $ 0<\lambda <1$,
$$\log B(s)=\int^s_{x=0}\frac{dx}{(x-1)(1-g(x))}= \int^s_{x=0}\left\{\frac{1}{x-1}+\frac{1}{x-1}\sum^\infty_{k=1}(g(x))^k\right\}dx.$$
Denote the definite integral of the second fraction by the function $G(s)$, namely
$$ G(s)=\int^s_{x=0}\left\{\frac{1}{x-1}\sum^\infty_{k=1}(g(x))^k\right\}dx,\quad \log \frac{B(s)}{1-s}=G(s),\quad B(s)=(1-s)e^{G(s)}. $$
Remark that, the function $g(x)$ is the exponential generating function of the sequence of derivatives
$$(g'(1),\quad g''(1),...,g^{(n)}(1),...),\quad g^{(n)}(1)=\frac{\lambda}{1-\lambda}\frac{(\lambda)^n}{n+1}. $$
and the sequence of constantes is denoted by
$$g_\bullet=\left(g_1=\frac{1}{2}, g_2=\frac{1}{3},..., g_n=\frac{1}{n+1},...\right). $$
Then
$$B_{n,k}((g'(1),\quad g''(1),...,g^{(n)}(1),...))=\left(\frac{\lambda}{1-\lambda}\right)^k \lambda^n B_{n.k}(g_\bullet). $$
The power of any exponential generating function is given by the Exponential Bell polynomials as follows,
$$(g(x))^k=\left(\frac{\lambda}{1-\lambda}\right)^k k!\sum^\infty_{n=k}\lambda^n B_{n,k}(g_\bullet)\frac{(x-1)^n}{n!}. $$
The change of summation order in the double sum leads to
$$\frac{1}{1-g(x)}=1+\sum^\infty_{n=1}(\lambda)^n \frac{(x-1)^n}{n!}\sum^n_{k=1}\left(\frac{\lambda}{1-\lambda}\right)^k k!B_{n,k}(g_\bullet) .$$
The exponential Bell polynomials $B_{n,k}(g_\bullet)$ are given by the Stirling numbers of the second kind $S(n+j,j)$ as folows
$$B_{n,k}(g_\bullet)=\frac{n!}{(n+k)!}\sum^k_{j=0}(-1)^{k-j}\binom{n+k}{k-j}S(n+j,j). $$
see  \cite{Fe}, \cite{WW} .

\textbf{The second method} is based on using the reciprocal power series in the neighborhood of the point $s=1$  given as follows,
$$\frac{1}{1-g(s)}=1+b(s),\quad b(s)=\sum^\infty_{n=1}b_n \frac{(s-1)^n}{n!},\quad (1-g(s))(1+b(s))=1, $$
where
$$b(s)=\sum^\infty_{k=1}(g(s))^k, \quad g(s)=\sum^\infty_{n=1}g^{(n)}(1) \frac{(s-1)^n}{n!},\quad g^{(n)}(1) =\left(\frac{\lambda}{1-\lambda}\right)\left(\frac{\lambda^n}{n+1}\right). $$
The  derivatives of the equality  $(1-g(s))(1+b(s))=1 $ result in
$$\sum^n_{k=0}\binom{n}{k}(1-g(s))^{(k)} (1+b(s))^{(n-k)}=0,\quad n=1,2,... $$
Equivalently,
$$(1-g(s))b_n+\sum^{n-1}_{k=1}\binom{n}{k}(1-g(s))^{(k)} (1+b(s))^{(n-k)}+(-g^{(n)}(s))(1+b(s))=0,\quad n=1,2,... $$
and at the point $s=1$ we have
$$b_n-\sum^{n-1}_{k=1}\binom{n}{k}g^{(k)}(1)b_{n-k}+(-g^{(n)}(1))=0,\quad (1-g)^{(0)}(1) =1,\quad b_0=1,$$
and the recurrence relation is specified as
$$b_n=\sum^n_{k=1}\binom{n}{k}g^{(k)}(1) b_{n-k}:=\left(\frac{\lambda}{1-\lambda}\right)\sum^n_{k=1}\binom{n}{k}\left(\frac{\lambda^k}{k+1}\right) b_{n-k},$$
or equivalently calculated by combinatorics
$$ b_n=\lambda^n\sum^n_{k=1}\left(\frac{\lambda}{1-\lambda}\right)^k  k! B_{n,k}(g_\bullet).  $$
In particular,
$$ b_1=\lambda\left(\frac{\lambda}{1-\lambda}\right)\frac{1}{2}, \quad b_2=  \lambda^2\left(\frac{\lambda}{1-\lambda}\right)\left(\frac{1}{3}+ \frac{1}{2} \frac{\lambda}{1-\lambda}\right),$$
$$ b_3= \lambda^3 \left(\frac{\lambda}{1-\lambda}\right)\left(\frac{1}{4}+ \frac{\lambda}{1-\lambda}+\frac{3}{4}\left( \frac{\lambda}{1-\lambda}\right)^2 \right ). $$

\begin{theorem}
	The Taylor series expansion of the function $G(s)$ in the neighborhood of $s=1$ is writen as follows
	$$G(s)=\sum^\infty_{n=1}\frac{(\lambda(s-1))^n-(-\lambda)^n}{n}\frac{1}{n!}\sum^n_{k=1}\left(\frac{\lambda}{1-\lambda}\right)^k k!B_{n,k}(g_\bullet) $$
	where $G(0)=0$ and
	$$ G(1)=-\sum^\infty_{n=1}\frac{(-\lambda)^n}{n}\frac{1}{n!}\sum^n_{k=1}\left(\frac{\lambda}{1-\lambda}\right)^k k!B_{n,k}(g_\bullet) ,\quad G^{(n)}(1)=\frac{\lambda ^n}{n}\sum^n_{k=1}\left(\frac{\lambda}{1-\lambda}\right)^k k!B_{n,k}(g_\bullet).$$
	Equivalently,
	$$G(s)=G(1)+\sum^\infty_{n=1}\frac{b_n}{n}\frac{(s-1)^n}{n!},\quad G(0)=0,\quad G(1)= \sum^\infty_{n=1}\frac{(-1)^{n-1} b_n}{n. n!}>0.$$
	
\end{theorem}
The integration is obvious following the definition,
$$G(s)= \int^s_{x=0}\sum^\infty_{n=1}(\lambda)^n (x-1)^{n-1}dx\sum^n_{k=1}\left(\frac{\lambda}{1-\lambda}\right)^k \frac{k!}{n!}B_{n,k}(g_\bullet).$$
We only remark the boundary conditions, knowing the values of $B_{n,k}(g_\bullet)$, and especially $B_{1,1}(g_\bullet)=\frac{1}{2}  $,
$$ B(0)=1,\quad B(s)=(1-s)e^{G(s)},\quad G'(0)=1-e^\lambda (1-\lambda),\quad  G'(1)=\frac{\lambda^2}{2(1-\lambda)} .$$
Remark that the sequence $\left(\frac{b_n}{n. n!}, n=1, 2,..., \right)$  is decreasing and the alternating summation geves 
$$G(1)= \sum^\infty_{n=1}\frac{(-1)^{n-1} b_n}{n. n!}=\left(b_1- \frac{ b_2}{2. 2!}\right) +\left(\frac{b_3}{3. 3!}-\frac{ b_4}{4. 4!} \right)+...>0.$$
A graphical presentation of some results of $B(s)$ are shown in Figure \ref{fig:G12}.
\begin{figure}[t]
	\includegraphics[scale=.5]{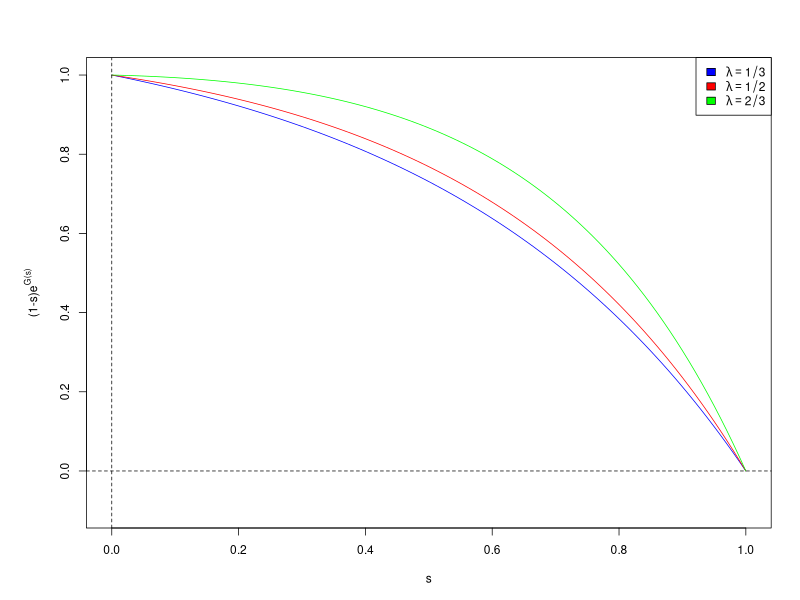} %Bern(.eps,.pdf)
	\caption{%
		Graphics of $B(s)=(1-s)e^{G(s)}$ for $\lambda=1/3$(blue), $\lambda=1/2$(red) and $\lambda=2/3(green)$.
	}\label{fig:G12} %% no full stop at the end
\end{figure}
The comparison of the Figure \ref{fig:Bs} and Figure \ref{fig:G12} shows the equivalence of the graphics for the function B(s) calculated by different methods.

In the following theorem, the  integration by parts in the neighbourhood of $s=1$ shows the explicit representation of the Koenigs function $B(s)$ as the series expansion of the special function $\rm{ Ei}(s) $ and $\rm {Ein}(s)$.

\begin{theorem}
	The function $B(s)$ is defined in the neighbourhood of the point $s=1$  by the special function exponential integral $\rm{ Ei}(\lambda k (x-1)) $ as follows,
	
	$1.$ first method, $$ B(s)=\left( (1-s)e^{D(s)}  \right)^{1-\lambda}, \quad D(0)=0 .$$
	where the function $D(s)$ is
	$$D(s)=\sum^\infty_{n=1}\sum^n_{k=0}\binom{n}{k}(-1)^{n-k}\left\{\sum^n_{j=1}\frac{-(\lambda k)^{j-1}e^{\lambda k (s-1)} }{[n]_{j\downarrow} (s-1)^{n+1-j}} +\frac{(\lambda k)^n \rm{Ei}(\lambda k (s-1))}{n!}\right\}.$$
	
	$2.$ second method, $$ B(s)= (1-s)e^{G(s)},\quad G(0)=0,   $$	
	where the function $G(s)$ is equal to the following sum
	$$ \sum^\infty_{n=1}\left(\frac{\lambda}{1-\lambda}\right)^n\sum^n_{k=0} \binom{n}{k} \lambda (-1)^{n-k} \sum^k_{j=0} \binom{k}{j}  (-1)^{k-j}\left\{\sum^k_{l=1}\frac{-( j)^{l-1}e^{\lambda j (s-1)} }{[k]_{l \downarrow}(\lambda (s-1))^{k+1-l }} +\frac{( j)^k \rm{Ei}(\lambda j (s-1))}{k!}\right\}.$$
\end{theorem}
\begin{proof}
	$1.$  first method,  By definition
	$$\log B(s)=\int^s_{x=0}\frac{(1-\lambda) dx}{x-e^{\lambda(x-1)}}=\int^s_{x=0}\frac{(1-\lambda) dx}{x-1+1-e^{\lambda(x-1)}}= \int^s_{x=0}\frac{(1-\lambda)dx}{(x-1)\left(1-\frac{e^{\lambda(x-1)}-1}{x-1}\right)}$$
	$$=(1-\lambda)\log(1-s)+(1-\lambda)\sum^\infty_{n=1}\int^s_{x=0}\frac{(e^{\lambda (x-1)}-1)^n dx}{(x-1)^{n+1}}.$$
	The indefinite integral for $n=1,2,...$ and  Isaac Newton binomial expansion
	$$\int\frac{(e^{\lambda(x-1)}-1)^n dx}{(x-1)^{n+1}}= (-1)^n\int\frac{dx}{(x-1)^{n+1}} + \sum^n_{k=1}\binom{n}{k}(-1)^{n-k}\int\frac{e^{\lambda k(x-1)} dx}{(x-1)^{n+1}}$$
	$$=\frac{(-1)^n}{(-n)(x-1)^n} +\sum^n_{k=1}\binom{n}{k}(-1)^{n-k}\left\{\sum^n_{j=1}\frac{-(\lambda k)^{j-1}e^{\lambda k (x-1)} }{[n]_{j\downarrow} (x-1)^{n+1-j}}+\frac{(\lambda k)^n \rm{Ei}(\lambda k (x-1))}{n!}\right\}.$$
	
	$2.$  second method,  By definition, we remamber the function $g(s)$ and the representation
	
	$$\frac{ f'(1)}{f(x)}=\frac{\lambda-1}{e^{\lambda (x-1)}-x}=\frac{1}{(x-1)(1-g(x))},\quad g(x)=\frac{\lambda}{1-\lambda}\left(\frac{e^{\lambda( x-1)}-1}{\lambda(x-1)} -1\right)$$
	The function $G(s)$, is defined
	$$ G(s)=\int^s_{x=0}\left\{\frac{1}{x-1}\sum^\infty_{n=1}(g(x))^n\right\}dx,$$
	where
	$$ \frac{((g(x))^n}{x-1}=\left(\frac{\lambda}{1-\lambda}\right)^n\sum^n_{k=0} \binom{n}{k} \frac{\lambda (-1)^{n-k}}{ (\lambda (x-1))^{k+1}} \sum^k_{j=0} \binom{k}{j}  (-1)^{k-j}e^{\lambda (x-1)j}  $$
	Finaly,
	$$ G(s)=\sum^\infty_{n=1}\left(\frac{\lambda}{1-\lambda}\right)^n\sum^n_{k=0} \binom{n}{k} \lambda (-1)^{n-k} \sum^k_{j=0} \binom{k}{j}  (-1)^{k-j}\int^s_{x=0}\frac{e^{\lambda (x-1)j} dx}{ (\lambda (x-1))^{k+1}} .  $$
	
\end{proof}
In particular, for $n=1$
$$\int^s_{x=0}\frac{g(x)dx}{(x-1)}=\int^s_{x=0}\left(\frac{e^{\lambda(x-1)}-1-\lambda(x-1)}{(1-\lambda)(x-1)^2}\right)dx $$
$$=\left(\frac{e^{\lambda(s-1)}-1}{(1-\lambda)}\right)\left(\frac{-1}{s-1}\right)- \frac{\lambda}{1-\lambda}\rm{Ein}(\lambda(1-s)).  $$
The function $\rm{Ei(x)} $ is negative and decreasing in the interval $ -\infty <x<0$. The values of the function $\rm{ Ei(\lambda k( s-1))}$ don't touch the branch point zero, $\rm{Ei(x)},  x \neq 0$. The initial condition $G(0)$ is attained because of the alternating summation. 
The graphical presentation of $G(s)$ and  $B(s)=(1-s)e^{G(s)}$ is shown in Fig. \ref{fig:GS}.
\begin{figure}[t]
	\includegraphics[scale=.5]{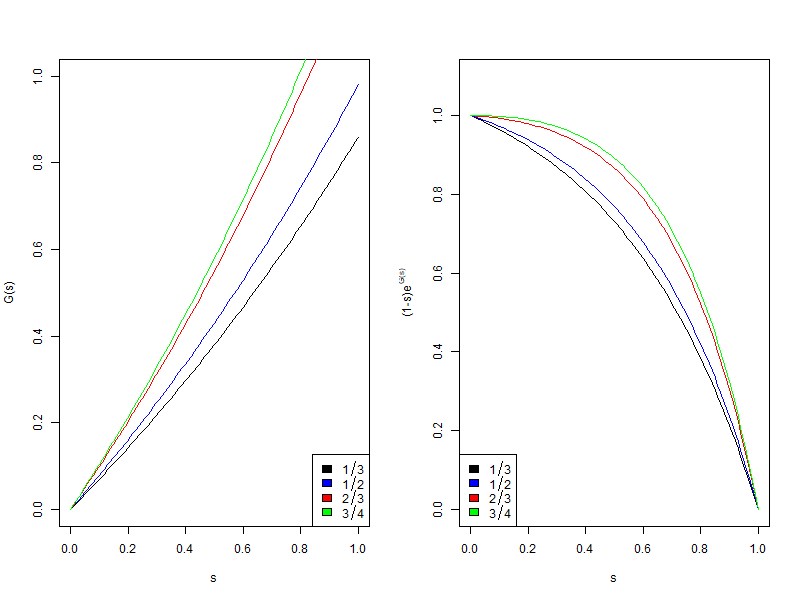} %Bern(.eps,.pdf)
	\caption{%
		Graphics of $G(s)$ (left) and $B(s)=(1-s)e^{G(s)}$ (right).
	}\label{fig:GS} %% no full stop at the end
\end{figure}	

\section{Backward Kolmogrov equation for the Poisson reproduction in critical case}
Now, we consider the critical case, where $ \lambda=1,\quad h(s)=e^{s-1}.$
The equation $ h(s)=s,\quad e^s=e s$ has a double root $ s=1$ and the Denjoy-Wolfs points given by $ -1=W(-e^{-1})$ coincide.
\subsection{In the neighborhood of the point $s=0$}
Let us consider the function $U(s), U(0)=0,$ defined by the integral for $ \lambda=1$ as
$$ U(s)=\int^s_{x=0} \frac{K dx}{f(x)}=\int^s_{x=0} \frac{dx}{e^{x-1}-x},\quad U'(s)=\frac{1}{e^{s-1}-s}\quad U'(0)=\frac{1}{e^{-1}}=e.$$
In the same notation $p(s)$ and $a(s)$ as for the subcritical case, we take $ \lambda=1$.
Then
$$-a_n=n(1-e)a_{n-1}+\sum^n_{k=2}\binom{n}{k} a_{n-k},\quad a_n=\sum^n_{k=1}B_{n,k}(p_\bullet)(-1)^k k!.$$
The function $p(s)$ is the exponential generating function of the sequence
$$ p_\bullet=(1-e, 1,1,...)=(1,1,...)+(-e,0,0,...)$$
In particular,
$$a_1=-p_1=e-1,\quad a_2=-1+2!(1-e)^2,\quad a_3=-1+3!(1-e)-3!(1-e)^3, $$
$$a_4=-1+4.2!(1-e)+3!(1-6(1-e)^2)+4!(1-e)^4, $$
$$a_5=-1+2!\{5(1-e)+10\}-3!\{10(1-e)^2+15(1-e)\}+4!10(1-e)^3-5!(1-e)^5. $$
Following the previous Example \ref{Example1}, the tabulated values of $ B_{n,k}(p_\bullet)$ can be obtained after direct substitution of $\lambda=1$ in the matrix (\ref{eq:lambda-matrix}).
$$U(s)=e\int^s_{x=0}\left\{1+\sum^\infty_{k=1}a_k \frac{x^k}{k!} \right\}dx =e \left\{s+\sum^\infty_{k=2}a_{k-1} \frac{s^{k} }{k!} \right\}. $$
It means,
$$U(0)=0,\quad U'(0)=e , \quad U^{(n)}(0)= e  a_{n-1},\quad n=2,3,....$$
Using the Faa Di Bruno formula, we confirm the derivative of the Koenigs function $C(s)$ from the previous section.
$$C(s)=\exp(U(s)),\quad C^{(n)}(s)=\exp(U(s))\sum^n_{k=0}B_{n,k}(U'(s), U''(s),...). $$
Finally, we note that for them is valid that, $C(0)=1, C'(0)=e$, and
$$C'(0)=U'(0),\quad C''(0)=(U'(0))^2+U''(0),\quad C^{(3)}(0)=(U'(0))^3+3U'(0)U''(0)+U^{(3)}(0). $$
Also, see the Example $7$, \cite{AP}, (with a  difference in the notation).
\begin{remark}\rm
Remamber the functions
$$A(s)=\int^s_{x=0} \frac{ (\lambda -1)dx}{e^{-\lambda}e^{\lambda x}-x},\quad U(s)=\int^s_{x=0} \frac{dx}{e^{x-1}-x}.$$
Denote the function $U_{\lambda}(s)=\frac{A(s) }{\lambda-1}$ as follows,
 $$U_{\lambda}(s)=\frac{A(s) }{\lambda-1},\quad U_{\lambda}(s)=\int^s_{x=0} \frac{dx}{e^{-\lambda}e^{\lambda x}-x},$$
Then for the critical case, the function $U_{1}(s):=U(s)$.
\end{remark}
The graphics of the functions $U_\lambda(s)$ and $U(s)$ and respectively  $C_\lambda (s)=e^{U_\lambda(s)}$ and $C(s)=e^{U(s)}$ are shown in Figure \ref*{fig:CS}.

\begin{figure}[t]
	\includegraphics[scale=.5]{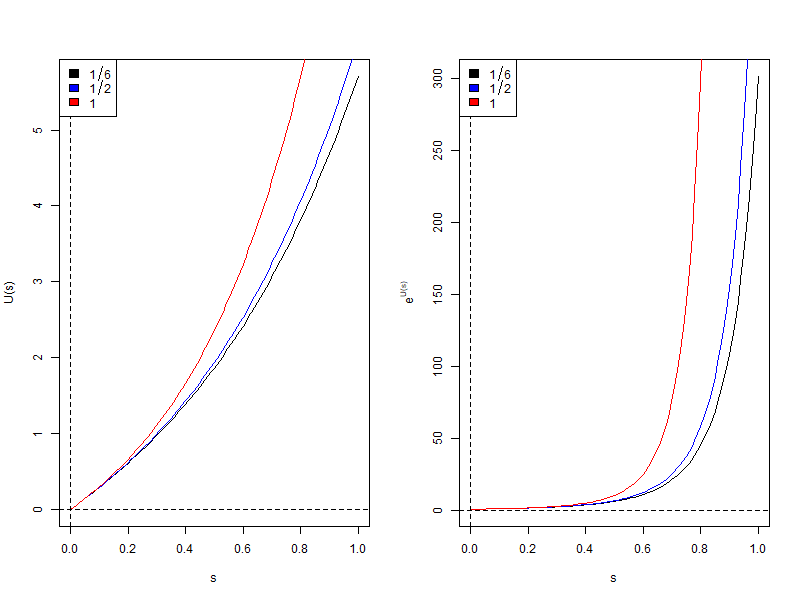} %Bern(.eps,.pdf)
	\caption{%
		Graphics of $U_\lambda(s)$ and $U(s), U(0)=0, U'(0)=e$ (left) and respectively  $C_\lambda (s)=e^{U_\lambda(s)}$ and $C(s)=e^{U(s)}, C(0)=1$ (right). The dashed line on right figure points out the minimal value at $C(0)=1$. 
	}\label{fig:CS} %% no full stop at the end
\end{figure}

\begin{theorem}
	The explicit form of the function $U(s)$ for $\lambda=1$ is given by integration of the series expansion,
	$$U(s)=e\sum^\infty_{k=0}(-1)^k\sum^k_{j=0}\binom{k}{j}{e^{ j s}}(-1)^{k-j} \sum^{k-j}_{i=1}\left(\frac{e }{ j}\right)^i (1+se)^{k-j-i}[k-j]_{i\downarrow}$$
	$$-e\sum^\infty_{k=0}(-1)^k\sum^k_{j=0}\binom{k}{j}(-1)^{k-j} \sum^{k-j}_{i=1}\left(\frac{e }{j}\right)^i [k-j]_{i\downarrow}.$$
\end{theorem}

Compare with the geometric branching critical case, \cite{TM}, and \cite{TMT},
$m=1$ and $ 0<s<1$.
$$\frac{ K}{f(x)}=\frac{1}{(x-1)^2}+\frac{1}{1-x},\quad U(s)=-\log(1-s)+\frac{1}{1-s },\quad s\neq 1 .$$
\subsection{In the neighborhood of the point $s=1$}
Consider the following decomposition in the neighborhood of $s=1$,
$$f(s)=K\left(\sum^\infty_{n=0}\frac{(s-1)^n}{n!} -s\right)=K\left(1+(s-1)+\sum^\infty_{n=2}\frac{(s-1)^n}{n!}-s\right)=K\sum^\infty_{n=2}\frac{(s-1)^n}{n!}. $$
$$\lambda=1, \quad \log C(s)=\int^s_{x=0} \frac{ Kdx}{f(x)},\quad f'(1)=0, $$
$$\frac{ K}{f(x)}=\frac{2}{(x-1)^2(1+2\varphi(x))},\quad \varphi(x)=\sum^\infty_{n=1}\frac{1}{(n+1)(n+2)} \frac{(x-1)^n}{n!}.$$
Then, we consider the reciprocal power series of the functions $c(s)$ and $ \varphi(s)$ defined by the following
$$\frac{K}{f(s)}=\frac{1}{\sum^\infty_{n=2}\frac{(s-1)^n}{n!}}=\frac{2}{(s-1)^2}\left( \frac{1}{1+2\varphi(s)}\right)=\frac{2(1+c(s))}{(s-1)^2} $$
The functions $\varphi(s)$ and $c(s)$ are defined by their derivatives,
$$\varphi(s)=\sum^\infty_{n=1}\varphi_n\frac{(s-1)^n}{n!},\quad c(s)= \sum^\infty_{n=1}c_n\frac{(s-1)^n}{n!}.$$
The general Leibniz rule applied on the equality $$(1+2\varphi(s))\{1+c(s)\}=1$$ gives the following recurrent relation for $ n=1,2,...$, and $\varphi(1)=0,c(1)=0 $,
$$\sum^n_{k=0}\binom{n}{k}2\varphi_k c_{n-k}=0,\quad c_n=-2 \sum^n_{k=1}\binom{n}{k} \frac{c_{n-k}}{(k+1)(k+2)},$$
where
$$ \varphi_0=c_0=1,\quad \varphi_1=\frac{1}{6}, \quad \varphi_2=\frac{1}{12},...,\quad \varphi_k=\frac{c_{n-k}}{(k+1)(k+2)},... .$$
The sequence $\varphi_\bullet$ is defined by
$$\varphi_\bullet= \left(\frac{1}{6}, \frac{1}{12},...,\frac{1}{(k+1)(k+1)},...\right). $$
The coefficients $c_k$ are given by the relation,
$$c_0=1,\quad c_1=\frac{-1}{3},\quad c_2=\frac{1}{3}\frac{1}{3!},\quad c_3=\frac{1}{3!}\frac{1}{15},\quad c_4=\frac{-1}{3! 45},\quad c_5=\frac{-5}{3!21}\frac{1}{9},...$$
Then, the integral on the interval $ 0<s<1$ is
$$ \int\frac{K ds}{f(s)}=\int\left\{ \frac{2}{(s-1)^2}-\frac{1}{3(s-1)}+\frac{1}{3}\frac{1}{3!}\frac{2}{2!}+ \sum^\infty_{k=3}\frac{2c_k}{k!}(s-1)^{k-2}\right\} d(s-1) $$
$$ =\frac{-2}{s-1}-\frac{1}{3}\log|s-1|+ \frac{(s-1)}{18}+\frac{c_3}{3!}(s-1)^2+ \sum^\infty_{k=4}\frac{2c_k(s-1)^{k-1}}{k!(k-1)}.$$
Then
$$U(s) =\log\left\{\frac{1}{ \sqrt[3]{1-s}}\exp\left(\frac{2}{1-s} -\frac{1-s}{18}\right)\right\}+\sum^\infty_{k=3}\frac{c_k(s-1)^{k-1}}{k!(k-1)}. $$
and
$$C(s)=e^{U(s)}= \left\{\frac{1}{ \sqrt[3]{1-s}}\right\}\exp\left(\frac{2}{1-s} -\frac{1-s}{18}\right)\exp\left\{\sum^\infty_{k=3}\frac{c_k(s-1)^{k-1}}{k!(k-1)}\right\}.$$
The multiplicative constant insuring boundary condition $C(0)=1$ is calculated by the Bell polynomials $B_{n.k}(\varphi_{\bullet})$.

\section{Conclusion}
In this work, we are focused on MBP with Poisson reproduction of particles, applying three different methods proven equivalent. The obtained solutions outline the mechanism of efficient computation involving branching processes with Poisson reproduction. The graphics are generated using the reciprocal power series and recursive relation. The used series expansions are generated by applying combinatorial methods. It achieves a very high precision with a moderate number of terms, well below 100. The exponential Bell polynomials apply to the arrangement of the constants and derivatives.

The proposed methods have relations to many other applications. One of the studied models relies on the special function $\rm{Ei(x)}$, which is not new to many physical models. It was noted that for $ 0<\lambda <1/6$, the graphic of the function  $B(s)$ is almost equal to $1-s$. It means the process is a near-linear pure death process. Then, the solution to the used Schr\"{o}der's equation is
$$ B(F(t,s))=M(t) B(s),\quad 1-F(t,s)=M(t)(1-s),\quad F(t,s,)=1-M(t)+M(t) s.$$

Finally, we note that the continuity of the Koenigs function $B(s)$ at the point $s=1$ is confirmed by the graphics for different values of the parameter $0<\lambda <1$. 

\section*{Acknowledgments}

Assen Tchorbadjieff acknowledges partial support from the Centre of Excellence in Informatics and ICT under Grant No.~BG16RFPR002-1.014-0018-C01, financed by the Research, Innovation and Digitalization for Smart Transformation Programme 2021--2027 and co-financed by the European Union.

\end{document}